\documentclass[11pt]{article}
\usepackage{amsmath}
\usepackage{dsfont}
\usepackage{cmap}
\usepackage{amsmath,amssymb,amsthm,graphicx,mathrsfs}
\usepackage[numbers,sort&compress]{natbib}
\usepackage{color}
\usepackage{indentfirst}
\numberwithin{equation}{section}
\newcommand{\beq}{\begin{equation}}
	\newcommand{\enq}{\end{equation}}

\newtheorem{Def}{Definition}[section]
\newtheorem{Lem}[Def]{Lemma}
\newtheorem{Thm}[Def]{Theorem}

\newcommand{\benu}{\begin{enumerate}}
	\newcommand{\beqa}{\begin{eqnarray}}
		\newcommand{\beqan}{\begin{eqnarray*}}
			\newcommand{\eay}{\end{array}}
		\newcommand{\edm}{\end{displaymath}}
	\newcommand{\eenu}{\end{enumerate}}
\newcommand{\eeq}{\end{equation}}
\newcommand{\eeqa}{\end{eqnarray}}
\newcommand{\eeqan}{\end{eqnarray*}}

\newcommand{\br}{\begin{Remark}}
\newcommand{\er}{\end{Remark}}

\newcommand{\bqa}{\begin{eqnarray}}
\newcommand{\eqa}{\end{eqnarray}}
\newcommand{\bqw}{\begin{eqnarray*}}
\newcommand{\eqw}{\end{eqnarray*}}

\newcommand{\bea}{\begin{array}{cc}}
\newcommand{\ena}{\end{array}}

\begin{document}
\begin{center}

{\large \bf
Strong attractors for the nonclassical diffusion equation with fading memory in time-dependent spaces}\\
\vspace{0.20in}Yuming Qin$^{1,\ast}$ $\ $ Xiaoling Chen$^{2}$ $\ $ Ke Wang$^{1}$\\
\end{center}
$^{1}$ Department of  Mathematics, Institute for Nonlinear Science, Donghua University, Shanghai 201620, P. R. China, \\
$^{2}$ Department of  Mathematics,  Donghua University, Shanghai 201620, P. R. China.\\
\vspace{3mm}

\begin{abstract}
	In this paper, we discuss the long-time behavior of solutions to the nonclassical diffusion equation with fading memory when the nonlinear term $f$ fulfills the polynomial growth of arbitrary order and the external force $ g(x)\in L^{2}(\Omega)$. In the framework of time-dependent spaces, we verify the existence and uniqueness of strong solutions by the Galerkin method, then we obtain the existence of the time-dependent global attractor $\mathscr{A}=\{A_t\}_{t\in \mathbb{R}}$ in $\mathcal{M}_t^1$.
\end{abstract}

\vspace{3mm} \hspace{4mm}{\bf Key words:} Strong solutions, Time-dependent global attractors, Nonclassical diffusion equation, Fading memory, Time-dependent spaces.

\vspace{3mm} \hspace{4mm}{\bf 2010 Mathematics subject classifications:} 35B41, 35D35, 35K57.

\section{Introduction}
\setcounter{equation}{0}
\let\thefootnote\relax\footnote{*Corresponding author.}
\let\thefootnote\relax\footnote{E-mails: yuming@dhu.edu.cn (Y. Qin), chenxl@mail.dhu.edu.cn (X. Chen), kwang@dhu.edu.cn (K. Wang)}
In this article, we investigate the following nonclassical diffusion equation with fading memory
\begin{equation}\label{eq1}
\left\{
\begin{aligned}
&u_t - \varepsilon(t)\Delta{u_t}-\Delta{u}-\int_0^{\infty} k(s) \Delta{u(t-s)} \,ds+f(u)=g(x),\ &&\text{in }\Omega \times (\tau,+ \infty),\\
&u(x,t)=0, && x\in\partial \Omega,\ t\in \mathbb{R}, \\
&u(x,t)=u_\tau(x,t), && x\in\Omega,\ t\leq \tau,
\end{aligned}
\right.
\end{equation}
where $\Omega$ is a bounded domain in $\mathbb{R}^n$ with smooth boundary $\partial \Omega$. For any $\tau\in \mathbb{R}$, $u_\tau:\Omega \times (-\infty,\tau]\rightarrow \mathbb{R}$ is a given function and $ g(x)\in L^{2}(\Omega)$ is an external force. We make some presumptions for function $ \varepsilon(t)$, nonlinear term $f$ and the memory term.

$(\mathbf{H}_1)$: Assume that $ \varepsilon(\cdot)\in C^{1}(\mathbb{R}) $ is a decreasing bounded function satisfying
\begin{equation}\label{c1}
\lim_{t \to +\infty}\varepsilon(t) =0.
\end{equation}
Especially, there exists a constant $ L>0 $, such that
\begin{equation}\label{c2}
\sup_{t\in \mathbb{R}}[\lvert\varepsilon(t)\rvert+ \lvert\varepsilon^{\prime}(t)\rvert] \leq L.
\end{equation}

$(\mathbf{H}_2)$: Assume that the nonlinear term $ f\in C^{1}(\mathbb{R}),\ f(0)=0 $ and fulfills the polynomial growth of arbitrary order
\begin{equation}\label{c3}
\beta_1\lvert s\rvert^p-\gamma_1\leq f(s)s\leq \beta_2\lvert s\rvert^p+\gamma_2,\ p\geq2,
\end{equation}
and meets the dissipation condition
\begin{equation}\label{c4}
f^{\prime}(s)\geq-l,
\end{equation}
where $\beta_i,\gamma_i(i=1,2)$ and $l$ are positive constants.\\
Let $F(s)=\int_{\tau}^{s}f(r)\,dr$, it's known from (\ref{c3}) that there are positive constants $\widetilde{\beta_i},\widetilde{\gamma_i}(i=1,2)$ such that
\begin{equation}\label{c5}
	\widetilde{\beta_1}\lvert s\rvert^p-\widetilde{\gamma_1}\leq F(s)\leq \widetilde{\beta_2}\lvert s\rvert^p+\widetilde{\gamma_2}.
\end{equation}

$(\mathbf{H}_3)$: Assume that for any $s>0$, $k(\cdot)\in C^{2}(\mathbb{R}^{+}),\ k(s)\geq 0,\ k^{\prime}(s)\leq 0$, and $k(\infty)=\lim\limits_{s \to \infty}k(s)=0$.
Besides, we assume that the function $\mu(s)=-k^{\prime}(s)$ and satisfies
\begin{equation}
\mu\in C^{1}(\mathbb{R}^{+})\cap L^{1}(\mathbb{R}^{+}),\  \mu(s)\geq 0,\ \mu^{\prime}(s)\leq 0,\ \forall s\in \mathbb{R}^{+},
\end{equation}
\begin{equation}
\mu^{\prime}(s)+\delta\mu(s)\leq 0,\ \forall s\in \mathbb{R}^{+},
\end{equation}
where $\delta$ is a positive constant.

Aifantis \cite{a} first proposed the classical diffusion equation in 1980, later he realized that the conductive medium's properties would also have an effect on the equation. As a result, he established the nonclassical diffusion equation by taking into account the conductive medium's viscosity, elasticity, pressure and other factors.

The viscoelasticity of the conductive medium is taken into account by the form of a memory term in Eq.\eqref{eq1}, which is an extension of the nonclassical diffusion equation used in fluid mechanics, solid mechanics and heat conduction theory (see \cite{a,lm,pg}). Under the influence of the convolution of function $\Delta u$ and memory kernel $k(\cdot)$, the pace of energy dissipation of the dynamical system corresponding to Eq.\eqref{eq1} is faster than the nonclassical diffusion equation without fading memory.

When $\varepsilon(t)$ is a positive constant independent of $t$, some conclusions have been achieved on the dynamical behavior of the solutions to Eq.\eqref{eq1} under different assumptions. For autonomous systems, Wang et al \cite{wyz} investigated the existence and regularity of global attractors for Eq.\eqref{eq1} in weak topological space and strong topological space  when nonlinearity meets the critical condition. Wang and Wang \cite{ww} considered the existence of trajectory and global attractors for Eq.\eqref{eq1} by applying the method presented by Chepyzhov and Miranville \cite{cm1,cm2}. Ma et al \cite{mxz} discussed the global attractors for Eq.\eqref{eq1} in an unbounded domain in $\mathbb{R}^3$ when nonlinearity is critical. Zhang et al \cite{zwg} added term $\lambda u$ to Eq.\eqref{eq1} and proved the existence of global attractors in the space $H^2(\Omega)\cap H_0^1(\Omega)\times L^{2}_{\mu}(\mathbb{R}^{+};H^2(\Omega)\cap H_0^1(\Omega))$ when the nonlinearity satisfies arbitrary polynomial growth. Qin and Dong \cite{qdmw} verified the strong global attractors for Eq.\eqref{eq1} in $H^2(\Omega)\cap H_0^1(\Omega)\times L^{2}_{\mu}(\mathbb{R}^{+};H^2(\Omega)\cap H_0^1(\Omega))$ by the condition $(C)$. For non-autonomous systems, Wang and Zhong \cite{wz} developed the results in \cite{wyz} and further obtained the existence of uniform attractors. Wang et al \cite{wjz} explored the long-time dynamical behavior of Eq.\eqref{eq1} in strong topological space when nonlinearity is critical. Li and Wang \cite{lw} obtained the existence and topological structure of a compact uniform attractor for Eq.\eqref{eq1} with subcritical nonlinearity by means of the technique of asymptotic regularity estimate.

However, handling the Eq.\eqref{eq1} becomes more challenging if $\varepsilon(t)$ is a bounded function that decreases monotonically. This is due to the fact that the classical semigroup theory is no longer entirely applicable, despite the external force $g$ being independent of $t$. The dynamical system generated by Eq.\eqref{eq1} becomes non-autonomous case owing to $\varepsilon(t)$. To this end, Plinio et al \cite{pdt} defined the solution operator as a family of maps acting on a time-dependent family of spaces $X_t$, in which the norm depends entirely on time $t$, and proposed the concept of time-dependent global attractors. Thereafter some scholars launched new researches based on the notation of time-dependent attractors. For example, Zhu et al \cite{zxz} showed the existence of time-dependent global attractors in $\mathcal{H}_t$ for a class of nonclassical reaction–diffusion equations with the forcing term $g(x)\in H^{-1}(\Omega)$ and the nonlinearity $f$ satisfying the polynomial growth of arbitrary $p-1(p\geq2)$ order. Tang et al \cite{tzl} gained the well-posedness of Eq.\eqref{eq1} by using the nonclassical method of Faedo-Galerkin and analytical techniques.

As far as we are aware, no one has ever study the strong attractors for the nonclassical diffusion equation with fading memory in time-dependent spaces when the nonlinear term $f$ fulfills the polynomial growth of arbitrary order. Hence, we shall investigate this matter. In order to solve the aforementioned issue, we must overcome some difficulties. On one hand, the existence of term $-\Delta{u_t}$ means that we cannot get higher regularity of the solution. Thus we cannot use the compact Sobolev embedding to verify the asymptotic compactness of the process. On the other hand, since Eq.\eqref{eq1} contains the function $\varepsilon(t)$ and the memory term, we need to build up a time-dependent space and introduce a new variable as done in \cite{d}, then the solution space becomes more complicated. In order to overcome these difficulties, we apply the contractive function method. Finally we can establish the existence of the time-dependent attractor in strong topological space.

In Section 2, we introduce some notions and preliminaries. In Section3, we show the definition of strong solution to Eq.\eqref{eq1}, and proved the existence and uniqueness of the strong solution by the Galerkin method under assumptions $(\mathbf{H}_1)$-$(\mathbf{H}_3)$. In Section4, we obatin the existence of an absorbing set and establish the existence of the strong attractor in time-dependent spaces.

\section{Notions and Preliminaries}
As in \cite{d}, we introduce a new variable which reflects the past history of $u$:
\begin{equation}
	\eta^{t}(x,s)=\int_0^{s} u(x,t-r)\,dr,\ s\geq 0,
\end{equation}
then 
\begin{equation}
	\eta^{t}_t(x,s)=u(x,t)-\eta^{t}_s(x,s),\ s\geq 0.
\end{equation}
Integrating by parts in $s$ and using assumption $k(\infty)=0$, we can rewrite Eq.\eqref{eq1} as
\begin{equation}\label{eq2.1}
	\left\{
	\begin{aligned}
		&u_t -\varepsilon(t)\Delta{u_t}-\Delta{u}-\int_0^{\infty} \mu(s)\Delta{\eta^{t}(s)}ds+f(u)=g(x),\\
		&\eta^{t}_t=-\eta^{t}_s+u,\\
	\end{aligned}
	\right.
\end{equation}
with initial boundary conditions
\begin{equation}\label{eq2.2}
	\left\{
	\begin{aligned}
		&u(x,t)=0,&& x\in\partial\Omega,\ t\geq\tau,\\
		&\eta^{t}(x,s)=0,&& (x,s)\in\partial\Omega\times \mathbb{R}^{+},\ t\geq \tau,\\
		&u(x,\tau)=u_\tau(x),&& x\in\Omega,\\
		&\eta^{\tau}(x,s)=\eta_\tau(x,s)=\int_0^{s} u_\tau(x,\tau-r)\,dr,\ && (x,s)\in \Omega\times \mathbb{R}^{+}.
	\end{aligned}
	\right.
\end{equation}
Additionally, we suppose the existence of two positive constants $\mathscr{R}$ and $\varrho\leq\delta$ such that
$$\int_0^{\infty} e^{-\varrho s}\Vert \nabla u(-s) \Vert^{2} \,ds\leq\mathscr{R}.$$

As in \cite{ps}, we denote the inner product and norm on $L^{2}(\Omega)$ by $\langle\cdot ,\cdot\rangle$ and $\Vert\cdot\Vert$, respectively. Let $A=-\Delta$ with domain $D(A)=H^{1}_0(\Omega)\cap H^{2}(\Omega)$. For $0\leq s \leq 2$, we consider a family of Hilbert spaces $H_s=D(A^{\frac{s}{2}})$, whose inner product and norm are defined as
\begin{gather*}
	\langle\cdot ,\cdot\rangle_{H_s}=\langle\cdot ,\cdot\rangle_{s}=\langle A^{\frac{s}{2}}\cdot ,A^{\frac{s}{2}}\cdot\rangle, \text{ and }\Vert\cdot \Vert_{H_s}=\Vert\cdot \Vert_{s}=\Vert A^{\frac{s}{2}}\cdot\Vert.
\end{gather*}
Obviously, $H_0=L^{2}(\Omega),\ H_1=H^{1}_0(\Omega) $ and $H_2=H^{1}_0(\Omega)\cap H^{2}(\Omega)$.\\
Particularly, we have the embeddings
\begin{center}
	$H_s\hookrightarrow H_r$, for $0\leq r<s\leq2$,
\end{center}
Now for $t\in \mathbb{R}$ and $0\leq\sigma\leq 1$, we introduce the time-dependent spaces $\mathcal{H}_t^{\sigma}=H_{\sigma+1}$, with the norm $$\Vert u\Vert^2_{\mathcal{H}_t^{\sigma}}=\Vert u\Vert^2_{\sigma}+\varepsilon(t)\Vert u\Vert^2_{\sigma+1}.$$
When $\sigma=0$, we can write $\Vert u\Vert^2_{\mathcal{H}_t^{0}}=\Vert u\Vert^2_{\mathcal{H}_t}$.\\
Due to \cite{cpt}, we have the compact embedding
$$\mathcal{H}_t^{\sigma}\hookrightarrow \mathcal{H}_t.$$

Then we consider the memory kernel $\mu(\cdot)$. Let $L^{2}_{\mu}(\mathbb{R}^{+};H_r)$ be the family of Hilbert spaces of functions $\varphi:\mathbb{R}^{+}\rightarrow H_r,\ 0<r<3$, endowed with the inner product and norm, respectively
$$\langle\varphi_1,\varphi_2\rangle_{\mu,r}=\int_0^{\infty} \mu(s)\langle\varphi_1,\varphi_2\rangle_{r} \,ds,$$
$$\Vert\varphi\Vert^2_{\mu,r}=\int_0^{\infty} \mu(s)\Vert\varphi\Vert^2_{r} \,ds.$$

Finally, the required time-dependent spaces can be described as $$\mathcal{M}_t^{\sigma}=\mathcal{H}_t^{\sigma}\times L^{2}_{\mu}(\mathbb{R}^{+};H_{\sigma+1})=H_{\sigma+1}\times L^{2}_{\mu}(\mathbb{R}^{+};H_{\sigma+1}),$$
with the norm $$\Vert z\Vert^2_{\mathcal{M}^\sigma_t}=\Vert (u,\eta^t)\Vert^2_{\mathcal{M}^\sigma_t}=\Vert u\Vert^2_\sigma+\varepsilon(t)\Vert u\Vert^2_{\sigma+1}+\Vert \eta^t\Vert^2_{\mu,\sigma+1}.$$
When $\sigma=0$, we can use $\mathcal{M}_t$ instead of $\mathcal{M}^0_t$.

Next we introduce some definitions and lemmas that are relevant to the problem we are studying. For $t\in \mathbb{R}$, let $X_t$ be a family of normed time-dependent spaces.
For every $t\in \mathbb{R}$, we introduce the $R$-ball of $X_t$,
$$\mathbb{B}_t(R)=\{z\in X_t:\Vert z\Vert_{X_t}\leq R\}.$$
We denote the Hausdorff semidistance of two (nonempty) sets $B,C\subset X_t$ by
$$\delta_t(B,C)=\sup_{x\in B}dist_{X_t}(x,C)=\sup_{x\in B}\inf_{y\in C}\Vert x-y\Vert_{X_t}.$$

\begin{Def} (\cite{cpt}) 
	A two-parameter family of mappings $\{ U(t,\tau):X_\tau \rightarrow X_t,\ t\geq\tau\in \mathbb{R}\}$ with following properties\\
	$(i)\ U(\tau,\tau)=Id$ is the identity map on $X_\tau,\ \tau\in \mathbb{R}$;\\
	$(ii)\ U(t,s)U(s,\tau)=U(t,\tau),\ \forall t\geq s\geq \tau$,\\
	is called a process.
\end{Def}

\begin{Def} (\cite{cpt}) 
	A family $\mathscr{C}=\{C_t\}_{t\in \mathbb{R}}$ of bounded sets $C_t \subset X_t$ is called uniformly bounded if there exists $R>0$ such that
	$$C_t \subset\mathbb{B}_t(R),\ \forall t\in \mathbb{R}.$$
\end{Def}

\begin{Def} (\cite{cpt}) 
	A uniformly bounded family $\mathscr{B}=\{B_t\}_{t\in \mathbb{R}}$ is called time-dependent absorbing set if for every $R>0$, there exists $t_0=t_0(R)$ such that $$U(t,\tau)\mathbb{B}_\tau(R)\subset B_t,\ for\ all\  t-\tau\geq t_0.$$
\end{Def}

\begin{Def} (\cite{cpt}) 
	A uniformly bounded family $\mathscr{K}=\{K_t\}_{t\in \mathbb{R}}$ is called pullback attracting if and only if the limit $$\lim_{\tau \to -\infty}\delta_t(U(t,\tau)C_\tau,K_t)=0$$
	holds for every uniformly bounded family $\mathscr{C}=\{C_t\}_{t\in \mathbb{R}}$ and every $t\in \mathbb{R}$.
\end{Def}

\begin{Def} (\cite{cpt}) 
	A time-dependent global attractor is defined as the smallest element of $\mathscr{A}=\{A_t\}_{t\in \mathbb{R}}$ such that\\
	(i) $A_t$ is compact in $X_t$;\\
	(ii) $\mathscr{A}$ is pullback attracting.
\end{Def}

\begin{Def} (\cite{cpt}) 
	We say that $\mathscr{A}=\{A_t\}_{t\in \mathbb{R}}$ is invariant if $$U(t,\tau)A_\tau=A_t,\ \forall t\geq \tau.$$
\end{Def}

\begin{Def} (\cite{zxz}) 
	The process $\{U(t,\tau)\}_{t\geq\tau}$ is said to be pullback asymptotically compact if for any $t\in \mathbb{R}$, any bounded sequence $\{x_n\}^\infty_{n=1}\subset X_{\tau_n}$ and any sequence $\{\tau_n\}^\infty_{n=1}$ with $\tau_n\rightarrow-\infty$ as $n\rightarrow\infty$, the sequence $\{U(t,\tau_n)x_n\}^\infty_{n=1}$ is precompact in $\{X_t\}_{t\in \mathbb{R}}$.
\end{Def}

\begin{Def}(\cite{zxz}) 
	Let $\{X_t\}_{t\in \mathbb{R}}$ be a family of Banach spaces and $\mathscr{C}=\{C_t\}_{t\in \mathbb{R}}$ be a family of uniformly bounded subset of $\{X_t\}_{t\in \mathbb{R}}$. We call a function $\psi^t_\tau(\cdot,\cdot)$, defined on $\{X_t\}_{t\in \mathbb{R}}\times\{X_t\}_{t\in \mathbb{R}}$, a contractive function on $C_\tau\times C_\tau$ if for fixed $t\in \mathbb{R}$ and any sequence $\{x_n\}^\infty_{n=1}\subset C_\tau$, there is a subsequence $\{x_{n_k}\}^\infty_{n=1}\subset\{x_n\}^\infty_{n=1}$ such that $$\lim_{k\to\infty}\lim_{l\to\infty}\psi^t_\tau(x_{n_k},x_{n_l})=0,\ for\ all\ t\geq\tau.$$
	We denote the set of all contractive functions on $C_\tau\times C_\tau$ by $\hat{C}(C_\tau)$.
\end{Def}

\begin{Lem}(\cite{wyz})\label{lem1}
	Let memory kernel $\mu(s)$ satisfies $(\mathbf{H}_3)$, then for any $\eta^t\in C(I;L^{2}_{\mu}(\mathbb{R}^{+};H_r)),0<r<3$, there exists a constant $\rho>0$, such that
	$$\langle\eta^t,\eta^t_s\rangle_{\mu,r}\geq \frac{\rho}{2}\Vert\eta^t\Vert^2_{\mu,r}.$$
\end{Lem}

\begin{Lem}(\cite{wm}) \label{lem2}
Assume that $X,B$ and $Y$ are Banach spaces with $X\subset\subset B$ and $B\subset Y$. Let $f_n$ be bounded in $L^\infty([0,T];X)$ and  $\frac{\partial f_n}{\partial t}$ is bounded in $L^p([0,T];Y)\ (p>1)$. Then, $f_n$ is relatively compact in $C([0,T];B)$.
\end{Lem}

\begin{Lem}(\cite{t}) \label{lem3}
Let $X$ and $Y$ be two Banach spaces such that $X\subset Y$ with a continuous injection. If a function $\phi$ belongs to $L^\infty([0,T];X)$ and is weakly continuous with values in $Y$, then $\phi$ is weakly continuous with values in $X$.
\end{Lem}

\begin{Lem}(\cite{zxz}) \label{lem4}
	Let $\{U(t,\tau)\}_{t\geq\tau}$ be a process on Banach spaces $\{X_t\}_{t\in \mathbb{R}}$, then $\{U(t,\tau)\}_{t\geq\tau}$ has a time-dependent global attractor in $\{X_t\}_{t\in \mathbb{R}}$ if the following conditions hold:\\
	(i) $\{U(t,\tau)\}_{t\geq\tau}$ has a pullback absorbing set $\mathscr{B}=\{B_t\}_{t\in \mathbb{R}}$ in $\{X_t\}_{t\in \mathbb{R}}$;\\
	(ii) $\{U(t,\tau)\}_{t\geq\tau}$ is pullback asymptotically compact.
\end{Lem}

\begin{Lem}(\cite{zxz}) \label{lem5}
	Let $\{U(t,\tau)\}_{t\geq\tau}$ be a process on Banach spaces $\{X_t\}_{t\in \mathbb{R}}$ and have a pullback absorbing set $\mathscr{B}=\{B_t\}_{t\in \mathbb{R}}$. Moreover, assume that, for any $\epsilon>0$, there exist $\tau_0=\tau_0(\epsilon)<t$ and $\psi^t_{\tau_0}(\cdot,\cdot)\in \hat{C}(B_{\tau_0})$ such that $$\Vert U(t,\tau_0)x-U(t,\tau_0)y\Vert_{X_t}\leq \epsilon+\psi^t_{\tau_0}(x,y),\ \forall x,y\in B_{\tau_0},\ for\ any\ t\in \mathbb{R}.$$ Then $\{U(t,\tau)\}_{t\geq\tau}$ is pullback asymptotically compact.
\end{Lem}

\section{Existence and uniqueness of strong solutions}
We begin by showing the definition of a strong solution to problem (\ref{eq2.1})-(\ref{eq2.2}).
\begin{Def}
	For any $T>\tau$, let $I=[\tau,T]$, a binary form $z=(u,\eta^t)\in \mathcal{M}_t^{1}$ is said to be a strong solution to problem (\ref{eq2.1})-(\ref{eq2.2}) in the time interval $I$, with the initial value $z_\tau=(u_\tau,\eta^\tau)\in \mathcal{M}_\tau^{1}$, if $z$ satisfies
	$$u\in C(I;\mathcal{H}_t^{1}),\ \eta^t\in C(I;L^{2}_{\mu}(\mathbb{R}^{+};H_2)),$$
	$$\eta^t_t+\eta^t_s\in L^\infty(I;L^{2}_{\mu}(\mathbb{R}^{+};H_1))\cap L^2(I;L^{2}_{\mu}(\mathbb{R}^{+};H_2)),$$
	and for all $\upsilon\in H_1,\ \phi\in L^{2}_{\mu}(\mathbb{R}^{+};H_1)$,
	\begin{equation*}
		\left\{
		\begin{aligned}
			 & \langle u_t,\upsilon\rangle+\varepsilon(t)\langle\nabla u_t,\nabla \upsilon\rangle+\langle\nabla u,\nabla \upsilon\rangle+\int_0^{\infty} \mu(s)\langle\nabla{\eta^{t}(s),\nabla \upsilon\rangle}ds+\langle f(u),\upsilon\rangle=\langle g,\upsilon\rangle, \\
			 & \langle \eta^t_t+\eta^t_s,\varphi\rangle_{\mu,1}=\langle u,\varphi\rangle_{\mu,1},
		\end{aligned}
	\right.
	\end{equation*}
hold $a.e.\ t\in I.$

We may now state the first important theory in this paper.
\end{Def}

\begin{Thm}\label{thm1}
Suppose that assumptions $(\mathbf{H}_1)$-$(\mathbf{H}_3)$ hold. For any initial data $z_\tau=(u_\tau,\eta^\tau)\in \mathcal{M}_\tau^1$, there exists a unique strong solution $z=(u,\eta^t)$ to problem \eqref{eq2.1}-\eqref{eq2.2} in $\mathcal{M}_t^1$.
Moreover, $z(t)$ is continuous with respect to the initial data. That is, let $z_i(t)(i=1,2)$ be two strong solutions to problem \eqref{eq2.1}-\eqref{eq2.2} with the initial datas $z_i(\tau)\in \mathcal{M}_\tau^1$, then there exists a positive constant $C$ such that
\begin{equation*}
	\Vert z_1(t)-z_2(t)\Vert^2_{\mathcal{M}_t^1}\leq e^{C(t-\tau)}\Vert z_1(\tau)-z_2(\tau)\Vert^2_{\mathcal{M}_\tau^1},\ t\geq \tau.
\end{equation*}
\end{Thm}

\begin{proof}
\subsection*{Proof of existence of strong solution needs the following three steps.}
$\mathbf{Step1:}$ First we choose the base functions $\{w_j\}_{j\geq1}$ in $H_2$ with $w_j$ being the eigenfunctions of the Laplacian operator subject to the Dirichlet boundary condition
	\begin{equation*}
	\left\{
	\begin{aligned}
		&-\Delta w_j=\lambda_jw_j,\\
		&w_j|_{\partial\Omega}=0,
	\end{aligned}
	\right.
\end{equation*}
where eigenvalues $\lambda_j$ satisfy $0<\lambda_1\leq\lambda_2\leq\cdots\cdots\leq\lambda_j\leq\cdots\cdots,\ \lambda_j\to\infty\ as\ j\to \infty.$\\
We also normalize $w_j$, then $\{w_j\}_{j\geq1}$ become an orthogonal basis in $\mathcal{H}_t^1$. At the same time, we choose an orthogonal basis $\{\zeta_j\}_{j\geq1}$ of $L^{2}_{\mu}(\mathbb{R}^{+};H_2)$.

Now we use the Faedo-Galerkin method to find the approximate solution. Let $m$ be a given positive integer and $$u_m=\sum_{i=1}^m a_{im}(t)w_i,\ \eta_m^t=\sum_{i=1}^m b_{im}(t)\zeta_i,$$ which satisfy the following indentities
\begin{equation}\label{eq3}
	\left\{
	\begin{aligned}
		&u_{mt} -\varepsilon(t)\Delta{u_{mt}}-\Delta{u_m}-\int_0^{\infty} \mu(s)\Delta{\eta^{t}_m(s)}ds+f(u_m)=g_m(x),\\
		&\eta^{t}_{mt}=-\eta^{t}_{ms}+u_m,\\
		&u_m(x,t)|_{\partial\Omega}=0,\ \eta^{t}_m(x,s)|_{\partial\Omega\times \mathbb{R}^{+}}=0,\ t\geq \tau,\\
		&u_m(x,\tau)=P_mu_\tau,\ \eta^{\tau}_m(x,s)=Q_m\eta_\tau,
	\end{aligned}
	\right.
\end{equation}
where $P_m,\ Q_m$ are projections on subspaces $Span\{w_1,\cdots,w_m\}\subset\mathcal{H}_t^1,\ Span\{\zeta_1,\cdots,\zeta_m\}\subset L^{2}_{\mu}(\mathbb{R}^{+};H_2)$ respectively and $g_m(x)=P_mg(x)$.

The fundamental existence theory of Ordinary Differential Equations states that there exists a unique solution $z_m=(u_m,\eta_m^t)$ to Eq.\eqref{eq3} on $[\tau,T]$.

$\mathbf{Step2:}$ We now try to get the priori estimates for the approximate solution $z_m$ obtained in the previous step. In order to proceed with the proof, we introduce the following lemma.
\begin{Lem}(\cite{tzl}) 
	Suppose that $(\mathbf{H}_1)-(\mathbf{H}_3)$ hold, then for any $t\in I$, we have the following estimates
	\begin{equation}\label{c6}
		\Vert u_m\Vert^2+\varepsilon(t)\Vert \nabla u_m\Vert^2+\Vert \eta^t_m\Vert^2_{\mu,1}+\int_\tau^{T}\left( \Vert \nabla u_m(r)\Vert^2+\Vert \eta^r_m\Vert^2_{\mu,1}+\Vert u_m(r)\Vert^p_p\right) dr\leq \mathcal{K}_1,
	\end{equation}
	\begin{equation}\label{c7}
		\Vert \nabla u_m\Vert^2+\Vert u_m\Vert^p_p\leq \mathcal{K}_2,
    \end{equation}
    \begin{equation}\label{c8}
    	\int_\tau^{T}\left( \Vert u_{mt}(r)\Vert^2+\varepsilon(t)\Vert \nabla u_{mt}(r)\Vert^2\right) dr\leq \mathcal{K}_3,
    \end{equation}
where $\mathcal{K}_1,\mathcal{K}_2,\mathcal{K}_3$ are positive constants only depending on $T$.
\end{Lem}

In fact, multiplying both sides of equation in (\ref{eq3}) by $-\Delta u_m$ and integrating it over $\Omega$, we get
\begin{align}\label{c9}
	\frac{d}{dt}&\left(\Vert \nabla u_m\Vert^2+\varepsilon(t)\Vert\Delta u_m\Vert^2+\Vert\eta^t_m\Vert^2_{\mu,2}\right)+(2-\varepsilon^{\prime}(t))\Vert \Delta u_m\Vert^2+2\langle\eta^t_m,\eta^t_{ms}\rangle_{\mu,2}\notag\\
	&=-2\langle f(u_m),-\Delta u_m\rangle+2\langle g_m,-\Delta u_m\rangle.
\end{align}
Using (\ref{c4}) and (\ref{c7}), we are able to determine that
\begin{align}\label{c10}
	-2\langle f(u_m),-\Delta u_m\rangle&=-2\langle f^\prime(u_m)\nabla u_m,\nabla u_m\rangle\notag\\
	&\leq 2l\Vert \nabla u_m\Vert^2\notag\\
	&\leq 2l\mathcal{K}_2.
\end{align}
By use of Young's inequality, we have
\begin{equation}\label{c11}
	2\langle g_m,-\Delta u_m\rangle\leq \Vert g_m\Vert^2+\Vert \Delta u_m\Vert^2.
\end{equation}
Substituting (\ref{c10}), (\ref{c11}) into (\ref{c9}) and by Lemma \ref{lem1}, we arrive at
\begin{align}\label{c12}
	\frac{d}{dt}&\left(\Vert \nabla u_m\Vert^2+\varepsilon(t)\Vert\Delta u_m\Vert^2+\Vert\eta^t_m\Vert^2_{\mu,2}\right)+(1-\varepsilon^{\prime}(t))\Vert \Delta u_m\Vert^2+\rho\Vert\eta_m^t\Vert^2_{\mu,2}\notag\\
	&\leq 2l\mathcal{K}_2+\Vert g_m\Vert^2.
\end{align}
Since $H_2\hookrightarrow H_1$, we can write $\widetilde{\lambda}\Vert \nabla u_m\Vert^2\leq\Vert \Delta u_m\Vert^2$, then (\ref{c12}) can be written as
\begin{align}\label{c13}
	\frac{d}{dt}&\left(\Vert \nabla u_m\Vert^2+\varepsilon(t)\Vert\Delta u_m\Vert^2+\Vert\eta^t_m\Vert^2_{\mu,2}\right)+\frac{\widetilde{\lambda}}{2}\Vert \nabla u_m\Vert^2+\frac{1}{2}\Vert \Delta u_m\Vert^2+\rho\Vert\eta_m^t\Vert^2_{\mu,2}\notag\\
	&\leq 2l\mathcal{K}_2+\Vert g_m\Vert^2.
\end{align}
Set $\alpha=\min\left\{\frac{\widetilde{\lambda}}{2},\frac{1}{2L},\rho\right\}$, we then deduce from (\ref{c13}) that
\begin{equation}
	\frac{d}{dt}\Vert z_m(t)\Vert^2_{\mathcal{M}^1_t}+\alpha\Vert z_m(t)\Vert^2_{\mathcal{M}^1_t}\leq 2l\mathcal{K}_2+\Vert g_m\Vert^2.
\end{equation}
Applying the $Gronwall$ lemma on $[\tau,t]$, we obtain
\begin{equation}\label{24}
	\Vert z_m(t)\Vert^2_{\mathcal{M}_t^1}\leq e^{-\alpha(t-\tau)}\Vert z_{m}(\tau)\Vert^2_{\mathcal{M}_\tau^1}+\mathfrak{C}_0.
\end{equation}
where $\mathfrak{C}_0=\frac{1}{\alpha}(1-e^{-\alpha t})(2l\mathcal{K}_2+\Vert g_m\Vert^2)$.\\
Therefore, $\{z_m\}_{m\geq1}$ is uniformly bounded in $L^\infty(I;\mathcal{M}_t^1)$, that is, $u_m$, $\eta_m^t$ is uniformly bounded in $L^\infty(I;\mathcal{H}_t^1)$, $L^\infty(I;L^{2}_{\mu}(\mathbb{R}^{+};H_2))$ respectively.

In addition, from assumption $(\mathbf{H}_1)$ and (\ref{c12}), we have
\begin{align}\label{c15}
	\frac{d}{dt}&\left(\Vert \nabla u_m\Vert^2+\varepsilon(t)\Vert\Delta u_m\Vert^2+\Vert\eta^t_m\Vert^2_{\mu,2}\right)+\Vert \Delta u_m\Vert^2+\rho\Vert\eta_m^t\Vert^2_{\mu,2}\notag\\
	&\leq 2l\mathcal{K}_2+\Vert g_m\Vert^2.
\end{align}
Integrating from $\tau$ to $T$ on both sides of (\ref{c15}), we get
\begin{align}\label{c14}
	\int_\tau^{T}\left( \Vert \Delta u_m(r)\Vert^2+\rho\Vert\eta_m^r\Vert^2_{\mu,2}\right) dr&\leq\Vert z_{m}(\tau)\Vert^2_{\mathcal{M}_\tau^1}+(2l\mathcal{K}_2+\Vert g_m\Vert^2)(T-\tau)\notag\\
	&\leq\Vert z_{m}(\tau)\Vert^2_{\mathcal{M}_\tau^1}+\mathfrak{C}_1,
\end{align}
where $\mathfrak{C}_1=(2l\mathcal{K}_2+\Vert g_m\Vert^2)(T-\tau)$.\\
Therefore, $\Delta u_m$, $\eta_m^t$ is uniformly bounded in $L^2(I;L^2(\Omega))$, $L^2(I;L^{2}_{\mu}(\mathbb{R}^{+};H_2))$ respectively.

We choose $q=\frac{p}{p-1}$, then in value of (\ref{c3}), we have
\begin{align}\label{c28}
	\int_\tau^{T}\int_\Omega \vert f(u_m(r))\vert^q dx dr&\leq \int_\tau^{T}\int_\Omega \left( \beta_2\vert u_m\vert^{p-1}+\gamma_2\right)^q  dx dr\notag\\
	&\leq C_{q,\beta_2}\int_\tau^{T}\Vert u_m(r)\Vert^p_p dr+C_{q,\gamma_2}\vert\Omega\vert(T-\tau).
\end{align}

Hence, $f(u_m)$ is uniformly bounded in $L^q(I;L^q(\Omega))$.

$\mathbf{Step3:}$ According to the findings in step2, we are aware that there is a subsequence of $z_m=(u_m,\eta_m^t)$, still denoted by $z_m$ suach that
\begin{align*}
	&u_m\rightharpoonup u\ weakly\ star\ in\ L^\infty(I;\mathcal{H}_t^1),\notag\\
	&u_m\rightharpoonup u\ weakly\ in\ L^2(I;H_2),\notag\\
	&u_{mt}\rightharpoonup u_t\ weakly\ in\ L^2(I;\mathcal{H}_t),\notag\\
	&\eta_m^t\rightharpoonup \eta^t\ weakly\ star\ in\ L^\infty(I;L^{2}_{\mu}(\mathbb{R}^{+};H_2)),\notag\\
	&\eta_m^t\rightharpoonup \eta^t\ weakly \ in\ L^2(I;L^{2}_{\mu}(\mathbb{R}^{+};H_2)),\notag\\
	&f(u_m)\rightharpoonup \mathcal{X}\ weakly\ in\ L^q(I;L^q(\Omega)).
\end{align*}
Taking limit for Eq.(\ref{eq3}), we find that $z=(u,\eta^t)$ is a solution to problem (\ref{eq2.1})-(\ref{eq2.2}) satisfying $$z\in L^\infty(I;\mathcal{M}_t^1).$$
By Lemma \ref{lem2} and the compact embedding $\mathcal{H}_t^{\sigma}\hookrightarrow \mathcal{H}_t$, we are able to acquire $u_m\longrightarrow u\in C(I;\mathcal{H}_t^\gamma)(0<\gamma<1)$.
Then Lemma \ref{lem3} allows us to reach the following conclusion $$u\in C(I;\mathcal{H}_t^1)\ for\ all\ t>\tau.$$
Indeed, it follows from the continuity of $f$ that $f(u_m)\longrightarrow f(u)$, a.e. in $\Omega\times I$, then we know $\mathcal{X}=f(u)$.

Finally, the continuity of $\eta^t$ can be proved by the same method as in \cite{zwg}, that is, $\eta^t\in C(I;L^{2}_{\mu}(\mathbb{R}^{+};H_2))$.

So far, we complete the proof of the existence of strong solution.

\subsection*{Proof of uniqueness of strong solution is the following.}
Indeed, let $z_i(\tau)\in \mathcal{M}_\tau^1(i=1,2)$ be two initial values and $z_i(t)$ be the corresponding solution to problem \eqref{eq2.1}-\eqref{eq2.2}. We define $\bar{z}(t)=z_1(t)-z_2(t)=(u_1-u_2,\eta^t_1-\eta^t_2)=(\bar{u},\bar{\eta}^t)$. Obviously, the following equation can be obtained
\begin{equation}\label{eq4}
	\left\{
	\begin{aligned}
		&\bar{u}_{t} -\varepsilon(t)\Delta{\bar{u}_{t}}-\Delta{\bar{u}}-\int_0^{\infty} \mu(s)\Delta{\bar{\eta}^{t}(s)}ds+f(u_1)-f(u_2)=0,\\
		&\bar{\eta}^{t}_{t}=-\bar{\eta}^{t}_{s}+\bar{u},\\
		&\bar{u}(x,t)|_{\partial\Omega}=0,\ \bar{\eta}^{t}(x,s)|_{\partial\Omega\times \mathbb{R}^{+}}=0,\ t\geq \tau,\\
		&\bar{u}(x,\tau)=u_1(\tau)-u_2(\tau),\ \bar{\eta}^{\tau}(x,s)=\eta_1^\tau-\eta_2^\tau,
	\end{aligned}
	\right.
\end{equation}
Multiplying both sides of equation in (\ref{eq4}) by $-\Delta\bar{u}$ and integrating it from $\Omega$, we get
\begin{align}\label{c19}
	\frac{d}{dt}&\left(\Vert \nabla \bar{u}\Vert^2+\varepsilon(t)\Vert\Delta \bar{u}\Vert^2+\Vert\bar{\eta}^t\Vert^2_{\mu,2}\right)+(2-\varepsilon^{\prime}(t))\Vert \Delta \bar{u}\Vert^2+\rho\Vert\bar{\eta}^t\Vert^2_{\mu,2}\notag\\
	&\leq-2\langle f(u_1)-f(u_2),-\Delta \bar{u}\rangle.
\end{align}
In light of (\ref{c4}), we know
\begin{align}\label{c20}
	-2\langle f(u_1)-f(u_2),-\Delta \bar{u}\rangle&\leq -2\langle f^\prime(\theta u_1+(1-\theta) u_2)\bar{u},-\Delta \bar{u}\rangle\notag\\
	&\leq 2l\Vert \nabla \bar{u}\Vert^2.
\end{align}
Using $\varepsilon^\prime(t)<0$ and (\ref{c20}), (\ref{c19}) can be written as
$$\frac{d}{dt}\Vert \bar{z}(t)\Vert^2_{\mathcal{M}_t^1}\leq 2l\Vert \nabla \bar{u}\Vert^2\leq 2l\Vert \bar{z}(t)\Vert^2_{\mathcal{M}_t^1}.$$
Applying the $Gronwall$ lemma on $[\tau,t]$, it's possible to draw the conclusion
\begin{equation*}
	\Vert \bar{z}(t)\Vert^2_{\mathcal{M}_t^1}\leq e^{2l(t-\tau)}\Vert \bar{z}(\tau)\Vert^2_{\mathcal{M}_\tau^1},
\end{equation*}
which proves the uniqueness of strong solutions. 

Thus, we complete the proof of Theorem \ref{thm1}
\end{proof}

\section{Time-dependent global attractors for strong solutions}
According to Theroy \ref{thm1}, we can define a continuous process $\{U(t,\tau)\}_{t\geq\tau}$ generated by Eqs.(\ref{eq2.1})-(\ref{eq2.2}) as $$U(t,\tau):\mathcal{M}_\tau^1\rightarrow \mathcal{M}_t^1,\ t\geq\tau\in \mathbb{R},$$ which acts as $$U(t,\tau)z_\tau=U(t,\tau)(u_\tau,\eta^\tau)=z(t)=(u,\eta^t).$$

\subsection{Absorbing sets in $\mathcal{M}_t^1$}
\begin{Lem}\label{lem6}
	Suppose that $(\mathbf{H}_1)-(\mathbf{H}_3)$ hold. If $z(t)$ is a strong solution to Eqs.(\ref{eq2.1})-(\ref{eq2.2}) with initial value $z_\tau\in\mathcal{M}_\tau^1$, then there exist a constant $Q\geq 0$ such that
	$$\Vert U(t,\tau)z_\tau\Vert^2_{\mathcal{M}_t^1}\leq e^{-\alpha(t-\tau)}\Vert z_\tau\Vert^2_{\mathcal{M}_\tau^1}+Q,\ \forall\ t\geq\tau.$$
\end{Lem}
\begin{proof}
	Multiplying both sides of equation in (\ref{eq2.1}) by $-\Delta u$ and repeating the arguments used in the proof of (\ref{24}), we easily get the result.
\end{proof}

\begin{Lem}\label{lem7}
	Suppose that $(\mathbf{H}_1)-(\mathbf{H}_3)$ hold, $z(t)$ is a strong solution to Eqs.(\ref{eq2.1})-(\ref{eq2.2}) with initial value $z_\tau\in\mathcal{M}_\tau^1$. There exists a constant $R_0>0$, such that the family $\mathscr{B}=\{\mathbb{B}_t(R_0)\}_{t\in \mathbb{R}}$ is a time-dependent absorbing set for the process $U(t,\tau)$.
\end{Lem}
\begin{proof}
	Let $R_0^2=2Q+1$, for any initial value $z_\tau\in \mathbb{B}_\tau(R)$, it follows from Lemma \ref{lem6} that
	\begin{equation*}
		\Vert U(t,\tau)z_\tau\Vert^2_{\mathcal{M}_t^1}\leq e^{-\alpha(t-\tau)}R^2 +Q.
	\end{equation*}
	Choosing $t_0=max\{0,\frac{1}{\alpha}\ln\frac{R^2}{1+Q}\}$, then for all $t\geq t_0+\tau$, we have
	\begin{equation}\label{l4}
		\Vert U(t,\tau)z_\tau\Vert^2_{\mathcal{M}_t^1}\leq 1+Q+Q=2Q+1=R_0^2.
	\end{equation}
	Letting $\mathscr{B}=\{\mathbb{B}_t(R_0)\}_{t\in \mathbb{R}}$, then
	\begin{equation*}
		U(t,\tau)\mathbb{B}_\tau(R)\subset \mathbb{B}_t(R_0),\ \forall t-\tau\geq t_0.
	\end{equation*} 
\end{proof}

\subsection{Time-dependent attractors}
Based on the above discussion, we give the most important conclusion of this article.
\begin{Thm}
	Under the same assumptions as in Lemma \ref{lem7}, the process $\{U(t,\tau)\}_{t\geq\tau}$ generated by Eqs.(\ref{eq2.1})-(\ref{eq2.2}) is pullback asymptotic compact in $\mathcal{M}_t^1$.
\end{Thm}
\begin{proof}
	Let $z_m=(u_m,\eta_m^t),\ z_n=(u_n,\eta_n^t)$ be the corresponding two solutions to Eqs.(\ref{eq2.1})-(\ref{eq2.2}) with initial data $z_m(\tau),\ z_n(\tau)$, set $\hat{z}=z_m-z_n=(u_m-u_n,\eta_m^t-\eta_n^t)=(\kappa,\zeta^t)$, then we infer from Eqs.(\ref{eq2.1})-(\ref{eq2.2}) that
	\begin{equation}\label{eq5}
		\left\{
		\begin{aligned}
			&\kappa_{t} -\varepsilon(t)\Delta{\kappa_{t}}-\Delta{\kappa}-\int_0^{\infty} \mu(s)\Delta{\zeta^{t}(s)}ds+f(u_m)-f(u_n)=0,\\
			&\zeta^{t}_{t}=-\zeta^{t}_{s}+\kappa,\\
			&\kappa(x,t)|_{\partial\Omega}=0,\ \zeta^{t}(x,s)|_{\partial\Omega\times \mathbb{R}^{+}}=0,\ t\geq \tau,\\
			&\kappa(x,\tau)=u_m(\tau)-u_n(\tau),\ \zeta^{\tau}(x,s)=\eta_m^\tau-\eta_n^\tau.
		\end{aligned}
		\right.
	\end{equation}
    Multiplying both sides of equation in (\ref{eq5}) by $-\Delta\kappa$ and integrating it from $\Omega$, we obtain
    \begin{align}\label{c25}
    	\frac{d}{dt}&\left(\Vert \nabla \kappa\Vert^2+\varepsilon(t)\Vert\Delta \kappa\Vert^2+\Vert\zeta^t\Vert^2_{\mu,2}\right)+(2-\varepsilon^{\prime}(t))\Vert \Delta \kappa\Vert^2+\rho\Vert\zeta^t\Vert^2_{\mu,2}\notag\\
    	&\leq-2\langle f(u_m)-f(u_n),-\Delta \kappa\rangle.
    \end{align}
    In view of (\ref{c4}), we know
    \begin{align}
    	-2\langle f(u_m)-f(u_n),-\Delta \kappa\rangle&\leq -2\langle f^\prime(\theta u_m+(1-\theta) u_n)\kappa,-\Delta \kappa\rangle\notag\\
    	&\leq 2l\Vert \nabla\kappa\Vert^2.
    \end{align}
    Due to the embedding $H_2\hookrightarrow H_1$, we can write $\widetilde{\lambda}\Vert \nabla \kappa\Vert^2\leq\Vert \Delta \kappa\Vert^2$, then (\ref{c25}) can be written as
    \begin{align}\label{c26}
    	\frac{d}{dt}&\left(\Vert \nabla \kappa\Vert^2+\varepsilon(t)\Vert\Delta \kappa\Vert^2+\Vert\zeta^t\Vert^2_{\mu,2}\right)+\widetilde{\lambda}\Vert \nabla \kappa\Vert^2+\Vert \Delta \kappa\Vert^2+\rho\Vert\zeta_m^t\Vert^2_{\mu,2}\notag\\
    	&\leq 2l\Vert \nabla\kappa\Vert^2.
    \end{align}
    Setting $\alpha=\min\left\{\widetilde{\lambda},\frac{1}{L},\rho\right\}$, we have 
    \begin{equation}
    	\frac{d}{dt}\Vert \hat{z}(t)\Vert^2_{\mathcal{M}^1_t}+\alpha\Vert \hat{z}(t)\Vert^2_{\mathcal{M}^1_t}\leq 2l\Vert \nabla\kappa\Vert^2.
    \end{equation}
    Applying the $Gronwall$ lemma on $[\tau,t]$, we obtain
    \begin{equation}\label{27}
    	\Vert \hat{z}(t)\Vert^2_{\mathcal{M}_t^1}\leq e^{-\alpha(t-\tau)}\Vert \hat{z}(\tau)\Vert^2_{\mathcal{M}_\tau^1}+2le^{-\alpha t}\int_\tau^{T} e^{\alpha r}\Vert \nabla\kappa(r)\Vert^2 dr,
    \end{equation}
    \begin{equation}
    	\Vert U(t,\tau)z_m(\tau)-U(t,\tau)z_n(\tau)\Vert^2_{\mathcal{M}_t^1}\leq e^{-\alpha(t-\tau)}\Vert z_m(\tau)-z_n(\tau)\Vert^2_{\mathcal{M}_\tau^1}+2l\int_\tau^{T} \Vert \nabla u_m(r)-\nabla u_n(r)\Vert^2 dr.
    \end{equation}

    On one hand, we set $\varPsi_\tau^t(z_m,z_n)=2l\int_\tau^{T} \Vert \nabla u_m(r)-\nabla u_n(r)\Vert^2 dr$, 
    due to the compact embedding $\mathcal{H}_t^{1}\hookrightarrow\mathcal{H}_t$, the boundness of $u_n$ in $\mathcal{H}_t^{1}$ and $u_n\in C(I;\mathcal{H}_t^{1})$, we easily know there is a subsequence $u_{nk}$ of $u_n$, which is convergent in $\mathcal{H}_t$, that is,
    $$\lim_{k\to\infty}\lim_{l\to\infty}\int_\tau^{T} \Vert u_{nk}(r)-u_{nl}(r)\Vert^2_1 dr=0.$$ 
    Thus, $\varPsi_\tau^t(z_m,z_n)$ is a contrative function.
    
    On the other hand, for any $\epsilon>0$, and any fixed $t\in \mathbb{R}$, set $\tau_0=t-\frac{1}{\alpha}\ln\frac{\Vert z_m(\tau)-z_n(\tau)\Vert^2_{\mathcal{M}_\tau^1}}{\epsilon^2}$, we are able to show without much difficulty that the process $\{U(t,\tau)\}_{t\geq\tau}$ is pullback asymptotically compact in $\mathcal{M}_t^1$ based on Lemma \ref{lem5}.
    
    Finally, combining Lemma \ref{lem4} and Lemma \ref{lem7}, we know that the process $\{U(t,\tau)\}_{t\geq\tau}$ has a time-dependent global attractor for strong solutions in $\mathcal{M}_t^1$.
\end{proof}

\section*{Acknowledgements}
This paper was in part supported by the National Natural Science Foundation of China with contract number 12171082, the fundamental research funds for the central universities with contract numbers 2232022G-13, 2232023G-13 and by a grant from science and technology commission of Shanghai municipality.

\end{document}